\documentclass[12pt]{amsart}
\usepackage{latexsym,amssymb,amsmath}

\numberwithin{equation}{section}
\usepackage{hyperref}
\usepackage{caption}
\usepackage{subcaption}
\usepackage{tikz}
\usepackage{array}
\usepackage{float}

\usetikzlibrary{positioning}
\usetikzlibrary{arrows}
\newtheorem{theorem}{Theorem}[section]
\newtheorem{lemma}[theorem]{Lemma}

\newtheorem{corollary}[theorem]{Corollary}

\theoremstyle{definition}
\newtheorem{definition}[theorem]{Definition} 
\newtheorem{remark}[theorem]{Remark}
\newtheorem{example}[theorem]{Example}

\newtheorem{discussion}[theorem]{Discussion}

\begin{document}


\newcommand{\del}{\operatorname{del}}
\newcommand{\lk}{\operatorname{lk}}
\newcommand{\reg}{\operatorname{reg}}
\newcommand{\pd}{\operatorname{pd}}

 
\title{Newton complementary duals of $f$-ideals}
\thanks{Version: \today}

\author{Samuel Budd}
\address{Department of Mathematics and Statistics\\
McMaster University, Hamilton, ON, L8S 4L8}
\curraddr{2061 Oliver Road, Thunder Bay, ON, P7G 1P7} 
\email{budds2@mcmaster.ca,sjbudd3@gmail.com  }

\author{Adam Van Tuyl}
\address{Department of Mathematics and Statistics\\
McMaster University, Hamilton, ON, L8S 4L8}
\email{vantuyl@math.mcmaster.ca}
 
\keywords{$f$-ideals, facet ideal, Stanley-Reisner correspondence,
$f$-vectors, Newton complementary dual, Kruskal-Katona}
\subjclass[2000]{05E45,05E40,13F55}

\begin{abstract}
A square-free monomial ideal $I$ of $k[x_1,\ldots,x_n]$ is
said to be an $f$-ideal if the facet complex and non-face complex
associated with $I$ have the same $f$-vector.    We show that 
$I$ is an $f$-ideal if and only if its Newton complementary dual
$\widehat{I}$ is also an $f$-ideal.  Because of this duality,
previous results about some classes of $f$-ideals can be extended
to a much larger class of $f$-ideals.   An interesting by-product of our
work is an alternative formulation of the Kruskal-Katona theorem
for $f$-vectors of simplicial complexes.
\end{abstract}
 
\maketitle


\section{Introduction} 

Let $I$ be a square-free monomial ideal of $R = k[x_1,\ldots,x_n]$ where
$k$ is a field.   Associated with any such ideal are two simplicial complexes.
The {\it non-face complex}, denoted $\delta_\mathcal{N}(I)$, (also called
the {\it Stanley-Reisner complex}) is the simplicial complex whose faces are
in one-to-one correspondence with the square-free monomials not in $I$.   
Faridi \cite{faridi} introduced a second complex, the {\it facet complex }
$\delta_{\mathcal{F}}(I)$, where the generators of $I$ define the 
facets of the simplicial complex (see the next section for complete definitions).
In general, the two simplicial complexes $\delta_\mathcal{N}(I)$ and $\delta_\mathcal{F}(I)$
can be very different.  For example, the two complexes may have different 
dimensions;  as a consequence, the {\it $f$-vectors} 
of $\delta_\mathcal{F}(I)$ and $\delta_\mathcal{N}(I)$,
which enumerate  all the faces of a given dimension,
may be quite different.

If $I$ is a square-free monomial ideal with the property
that the $f$-vectors of 
$\delta_\mathcal{F}(I)$ and $\delta_\mathcal{N}(I)$ are the same, then 
$I$ is called an {\it $f$-ideal}.   The notion of an $f$-ideal was first 
introduced by Abbasi, Ahmad, Anwar, and Baig 
\cite{abbasi}.   It is natural to ask if it is possible to
classify all the square-free monomial ideals that are $f$-ideals.  
Abbasi, et.\ al  classified all the $f$-ideals generated in degree two.   This result was later generalized by Anwar, Mahmood, Binyamin, and Zafar \cite{anwar} which classified all the $f$-ideals $I$ that are unmixed and generated in
degree $d \geq 2$.  An alternative proof for this result was found by Guo and 
Wu \cite{guo2}.  Gu, Wu, and Liu \cite{guo} later
removed the unmixed restriction
of \cite{anwar}. Other work related to $f$-ideals includes the papers
\cite{mahmood,mahmood2}.

The purpose of this note is to show that the property of being 
an $f$-ideal is preserved after taking the Newton complementary dual of
$I$.  The notion of a Newton complementary dual was first introduced in a more
general context by 
Costa and Simis \cite{costa} in their study of Cremona maps;  additional
properties were developed by D\'oria and Simis \cite{doria}.   Ansaldi, Lin, and
Shin \cite{ansladi} later investigated the Newton complementary duals of monomial ideals.
Using the definition of \cite{ansladi}, the {\it Newton complementary dual} of 
a square-free monomial ideal $I$ is
\[\widehat{I} = \left.\left\langle \frac{x_1\cdots x_n}{m} ~\right| 
m \in \mathcal{G}(I)
\right\rangle\]
where $\mathcal{G}(I)$ denotes the minimal generators of $I$.  With this notation, we 
prove:

\begin{theorem}[Theorem \ref{maintheorem2}] \label{maintheorem} 
Let $I \subseteq R$ be a square-free monomial ideal.
Then $I$ is an $f$-ideal if and only if $\widehat{I}$ is an $f$-ideal.
\end{theorem}

\noindent
Our proof involves relating the $f$-vectors of the four simplicial
complexes $\delta_\mathcal{F}(I)$,
$\delta_\mathcal{N}(I)$,  $\delta_\mathcal{F}(\widehat{I})$, and 
$\delta_\mathcal{N}(\widehat{I})$.
An interesting by-product of this discussion is to give a reformulation
of the celebrated Kruskal-Katona theorem (see \cite{ka,kr}) 
which classifies what vectors can
be the $f$-vector of a simplicial complex (see Theorem \ref{kk}).

A consequence  of Theorem \ref{maintheorem} is that $f$-ideals
come in ``pairs".  
Note that when $I$ is an $f$-ideal generated in degree $d$, then 
$\widehat{I}$ gives us an $f$-ideal generated in degree $n-d$.
We can use the classification of \cite{abbasi} of
$f$-ideals generated in degree  two to also give us a classification
of $f$-ideals generated in degree $n-2$.  
This corollary and others are given as applications
of Theorem \ref{maintheorem}.

Our paper uses the following outline.  In Section 2 we provide all the necessary
background results.   In Section 3, we introduce the Newton complementary
dual of a square-free monomial ideal, and we study how the
$f$-vector behaves under this duality.   In
Section 4 we prove  Theorem \ref{maintheorem}  and devote the rest of
the section to applications.

\noindent
{\bf Acknowledgments.} 
Parts of this paper appeared in the first author's MSc project \cite{B}.
The second author's research was
supported in part by NSERC Discovery Grant 2014-03898.
We thank Hasan Mahmood for his feedback and comments.  
{\tt Macaulay2} \cite{mac2}
was used in the initial stages of this project.  We also thank the
referee for his/her helpful suggestions and improvements.


\section{Background}

In this section, we review the required background results.   

Let $X = \{x_1,\ldots,x_n\}$ be a set of vertices.   A {\it simplicial complex} $\Delta$ on
$X$ is a subset of the power set of $X$ that satisfies: $(i)$ if $F \in \Delta$ and
$G \subseteq F$, then $G \in \Delta$, and $(ii)$ $\{x_i\} \in \Delta$ for 
$i = 1,\ldots,n$.   An element $F \in \Delta$ is called a {\it face};  
maximal faces with respect to inclusion are called {\it facets}.  
If $F_1,\ldots,F_r$ are the facets of $\Delta$, then 
we write $\Delta = \langle F_1,\ldots,F_r \rangle$.  

For any face $F \in \Delta$, the {\it dimension} of $F$ is given by
$\dim(F) = |F|-1$.   Note that $\emptyset \in \Delta$ and 
$\dim(\emptyset) = -1$.  
The {\it dimension} of $\Delta$ is given by 
$\dim(\Delta) = \max\{\dim(F) ~|~ F \in \Delta\}$.
If $d = \dim(\Delta)$, then the {\it $f$-vector} of $\Delta$ is the $d+2$ tuple
\[f(\Delta) = (f_{-1},f_0,f_1,\ldots,f_d)\]
where $f_i$ is number of faces of dimension $i$ in $\Delta$.   
We write $f_i(\Delta)$ if we need to specify the simplicial complex.

Suppose that $I$ is a square-free monomial ideal of $R = k[x_1,\ldots,x_n]$ with
$k$ a field.   We use $\mathcal{G}(I)$ to denote the unique set
of minimal generators of $I$.
If we identify the variables of $R$ with the vertices $X$, we can
associate to $I$ two simplicial complexes.  The {\it non-face complex} 
(or {\it Stanley-Reisner complex}) is the 
simplicial complex
\[\delta_\mathcal{N}(I) = \left \{ \{x_{i_1},\ldots,x_{i_j}\} \subseteq X ~|~
x_{i_1}\cdots x_{i_j} \not\in I\right\}.\]
In other words, the faces of $\delta_{\mathcal{N}}(I)$ are 
in one-to-one correspondence
with the square-free monomials of $R$ not in the ideal $I$.  The {\it facet complex}
is the simplicial complex
\[\delta_\mathcal{F}(I) = \langle \{x_{i_1},\ldots,x_{i_j}\} \subseteq X ~|~ x_{i_1}\cdots x_{i_j} 
\in \mathcal{G}(I) \rangle. \]
The facets of $\delta_\mathcal{F}(I)$ are in one-to-one 
correspondence with the minimal generator of $I$.

In general, the two simplicial complexes  
$\delta_\mathcal{N}(I)$ and $\delta_\mathcal{F}(I)$ 
constructed from $I$ 
are very different.    In this note, we are interested
in the following family of monomial ideals:

\begin{definition}
A square-free monomial ideal $I$ is an {\it $f$-ideal} if
$f(\delta_\mathcal{N}(I)) = f(\delta_\mathcal{F}(I))$.
\end{definition}

\begin{example}
\label{mixed}
We illustrate the above ideas with the following 
example.   Let $I=\langle x_1x_4, x_2x_5, x_1x_2x_3, x_3x_4x_5 
\rangle \subseteq R=k[x_1, x_2, x_3, x_4, x_5]$ be a square-free monomial ideal. Then Figure \ref{figure} shows both the
facet and non-face complexes that are associated with $I$.
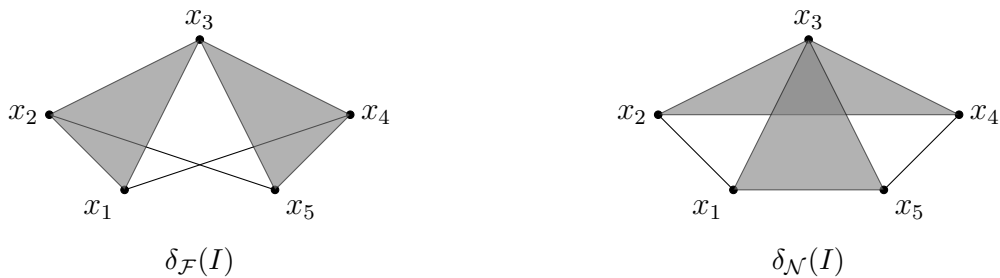
\begin{figure}[H]
    \centering
    \begin{subfigure}[h]{0.50\textwidth}
        \centering
        \begin{tikzpicture}
\draw[fill] (0,0) circle [radius=0.05];
\node [left] at (0,0) {$x_2$};
\draw[fill] (1,-1) circle [radius=0.05];
\node [below left] at (1,-1) {$x_1$};
\draw[fill] (2,1) circle [radius=0.05];
\node [above] at (2,1) {$x_3$};
\draw[fill] (3,-1) circle [radius=0.05];
\node [below right] at (3,-1) {$x_5$};
\draw[fill] (4,0) circle [radius=0.05];
\node [right] at (4,0) {$x_4$};
\path (3,-1) edge (0,0);
\path (1,-1) edge (4,0);
\draw[fill=gray, opacity=0.6] (0,0) -- (1,-1) -- (2,1) -- (0,0);
\draw[fill=gray, opacity=0.6] (2,1) -- (3,-1) -- (4,0) -- (2,1);
\end{tikzpicture}

        \caption*{$\delta_\mathcal{F} (I)$}
        
    \end{subfigure}
    ~ 
    \begin{subfigure}[h]{0.5\textwidth}
        \centering
        \begin{tikzpicture}
\draw[fill] (0,0) circle [radius=0.05];
\node [left] at (0,0) {$x_2$};
\draw[fill] (1,-1) circle [radius=0.05];
\node [below left] at (1,-1) {$x_1$};
\draw[fill] (2,1) circle [radius=0.05];
\node [above] at (2,1) {$x_3$};
\draw[fill] (3,-1) circle [radius=0.05];
\node [below right] at (3,-1) {$x_5$};
\draw[fill] (4,0) circle [radius=0.05];
\node [right] at (4,0) {$x_4$};
\path (1,-1) edge (0,0);
\path (3,-1) edge (4,0);
\draw[fill=gray, opacity=0.6] (0,0) -- (4,0) -- (2,1) -- (0,0);
\draw[fill=gray, opacity=0.6] (1,-1) -- (2,1) -- (3,-1) -- (1,-1);
\end{tikzpicture}
	        \caption*{$\delta_\mathcal{N}(I)$}
         \end{subfigure}
    \caption{Facet and non-face complexes of $I=\langle x_1x_4, x_2x_5, x_1x_2x_3, x_3x_4x_5 \rangle$.}
   \label{figure}
\end{figure}
\noindent
From Figure \ref{figure}, one can see that 
$f(\delta_\mathcal{F}(I))=f(\delta_\mathcal{N}(I))=(1,5,8,2)$, and therefore $I$ 
is an $f$-ideal.  We note that $I$ in this example is generated
 by monomials of different degrees.
In most of the other papers on this topic (e.g.,
\cite{abbasi,anwar,guo,guo2,mahmood,mahmood2}) the
focus has been on {\it equigenerated ideals}, i.e.,
ideals where are all generators have the same degree.
\end{example}

\begin{remark}\label{variablecase}
It is important to note that $\delta_\mathcal{F}(I)$
and $\delta_\mathcal{N}(I)$ may be simplicial complexes 
on different sets of vertices, and in particular, one must pay attention
to the ambient ring. 
  
For example, consider
 $I = \langle x_1,x_2x_3,x_2x_4,x_3x_4 \rangle \subseteq k[x_1,\ldots,
x_5]$.   For this ideal, $\delta_\mathcal{F}(I)$ is a simplicial
complex on $\{x_1,x_2,x_3,x_4\}$ with facets $\{\{x_1\},\{x_2,x_3\},
\{x_2,x_4\},\{x_3,x_4\}\}$.   So $f(\delta_\mathcal{F}(I)) = 
(1,4,3)$.   The vertices of 
$\delta_\mathcal{N}(I)$ are
$\{x_1,x_2,x_3,x_4,x_5\} \setminus \{x_1\}$.   Its facets 
are $\{\{x_2,x_5\},\{x_3,x_5\},\{x_4,x_5\}\}$.     From
this description, we see that $I$, in fact, is an $f$-ideal.

Note, however, that if $I$ is an $f$-ideal, and if every generator
of $I$ has degree at least two, then $\delta_{\mathcal{F}}(I)$ and 
$\delta_{\mathcal{N}}(I)$ must be simplicial complexes on
the vertex set $\{x_1,\ldots,x_n\}$.  To see why, since
every generator of $I$ has degree $\geq 2$, this implies that
$\{x_i\} \in \delta_{\mathcal{N}}(I)$ for all $i=1,\ldots,n$.  So,
$n = f_0(\delta_{\mathcal{N}}(I)) = f_0(\delta_{\mathcal{F}}(I))$,
that is, $\delta_{\mathcal{F}}(I)$ must also have $n$ vertices.

The above observation implies that
 the ideal  $I = \langle x_1x_2 \rangle \subseteq k[x_1,x_2,x_3]$
cannot be an $f$-ideal since it is generated by a monomial
of degree two, but 
$\delta_{\mathcal{F}}(I)$ is a simplicial complex on $\{x_1,x_2\}$,
but the vertices of $\delta_\mathcal{N}(I)$ are $\{x_1,x_2,x_3\}$.
\end{remark}


\begin{remark}\label{monomialpov}
Although the $f$-vector counts faces of a simplicial complex,
we can reinterpret the $f_j$'s
as counting square-free
monomials of a fixed degree.  In particular, 
\[f_j(\delta_\mathcal{N}(I)) = 
\#\left\{ m \in R_{j+1} ~\left|~
\begin{array}{c}
\mbox{$m$ is a square-free monomial of degree} \\
\mbox{$j+1$
and $m \not\in I_{j+1}$}
\end{array}
\right\}\right..\]
On the other hand, for the 
$f$-vector of $\delta_\mathcal{F}(I)$
we have
\[f_j(\delta_\mathcal{F}(I)) = \#\left\{m \in R_{j+1} ~\left|~
\begin{array}{c}
\mbox{$m$ is a square-free monomial of degree $j+1$ } \\
\mbox{that divides some $p \in \mathcal{G}(I)$}
\end{array}
\right\}\right..\]
Here, $R_t$, respectively $I_t$, denotes the degree $t$ homogeneous
elements of $R$, respectively $I$.
\end{remark}

We refine Remark \ref{monomialpov} by introducing a partition
of the set of square-free monomials of degree $d$.  This
partition will be useful in Section \ref{fvector}.    
For each integer $d \geq 0$, let $M_d \subseteq R_d$
denote the set of square-free monomial of degree
$d$ in $R_d$.   Given a square-free monomial ideal
$I$ with generating set $\mathcal{G}(I)$, set
\begin{eqnarray*}
A_d(I) &=&  \left\{m \in M_d ~\left|~ m \not\in I_d 
~\mbox{and $m$ does not divide any element of $\mathcal{G}(I)$} 
\right\}\right.\\
B_d(I) & =&  \left\{m \in M_d ~\left|~ m \not\in I_d 
~\mbox{and $m$ divides some element of $\mathcal{G}(I)$} 
\right\}\right.\\
C_d(I) & =& \{m \in M_d ~|~ m \in \mathcal{G}(I) \}, ~~\mbox{and} \\
D_d(I) & = & \{m \in M_d ~|~ m \in I_d \setminus \mathcal{G}(I)\}.
\end{eqnarray*}
So, for any square-free monomial ideal $I$ and integer
$d \geq 0$ we have
the partition
\begin{equation}\label{partition}
M_d = A_d(I) \sqcup B_d(I) \sqcup C_d(I) \sqcup D_d(I).
\end{equation}
Using this notation, we have the following characterization
of $f$-ideals.

\begin{lemma}\label{partitionchar}
Let $I \subseteq k[x_1,\ldots,x_n]$ be a square-free 
monomial ideal.  Then $I$ is an $f$-ideal if and only
if $|A_d(I)| = |C_d(I)|$ for all $0 \leq d \leq n$.
\end{lemma}

\begin{proof}
Note that Remark \ref{monomialpov} implies that 
\[f_j(\delta_\mathcal{N}(I)) = |A_{j+1}(I)| + |B_{j+1}(I)| ~~
\mbox{for all $j \geq -1$}\]
and 
\[f_j(\delta_\mathcal{F}(I)) = |B_{j+1}(I)| + |C_{j+1}(I)| ~~
\mbox{for all $j \geq -1$}.\]
The conclusion now follows since 
$f_j(\delta_\mathcal{F}(I)) = f_j(\delta_\mathcal{N}(I))$
for all $-1 \leq j \leq n-1$
if and only if $|A_d(I)| = |C_d(I)|$ for all $0 \leq d \leq n$.
\end{proof}


\section{The Newton complementary dual and $f$-vectors}

We introduce the generalized Newton complementary dual of a monomial
ideal as defined by \cite{ansladi} (based upon \cite{costa}).  
We then show
how the $f$-vector behaves under this operation.

\begin{definition}
Let $I \subseteq R=k[x_1,\ldots,x_n]$ be a monomial
ideal with $\mathcal{G}(I) = \{m_1,\ldots,m_p\}$.
Suppose that $m_i = x_1^{\alpha_{i,1}}x_2^{\alpha_{i,2}}
\cdots x_n^{\alpha_{i,n}}$
for all $i=1,\ldots, p$. 
Let $\beta=(\beta_1,\ldots,\beta_n)\in \mathbb{N}^n$
be a vector such that $\beta_{i} \geq \alpha_{k,l}$
for all $l=1,\ldots,n$ and $k=1,\ldots,p$.
The {\it generalized Newton complementary dual
of $I$} determined by $\beta$ is the ideal
\[
\widehat{I}^{[\beta]} = \left.\left\langle \frac{x^{\beta}}{m} ~\right|~ m\in \mathcal{G}(I)\right\rangle 
= \left\langle \frac{x^{\beta}}{m_1}, \frac{x^{\beta}}{m_2}, \ldots, \frac{x^{\beta}}{m_p}  \right\rangle
~~\mbox{
where $x^\beta = x_1^{\beta_1}\cdots x_n^{\beta_n}$.}\]
\end{definition}

\begin{remark}
If $I \subseteq R = k[x_1,\ldots,x_n]$ is a square-free
monomial ideal, then one can take $\beta = (1,\ldots,1) = 
{\bf 1}$, i.e., $x^\beta = x_1\cdots x_n$. 
 For simplicity, we denote $\widehat{I}^{[1]}$ 
by $\widehat{I}$ and call it the {\it complementary dual} of $I$.
Note that we have $\skew{5.2}\widehat{\widehat{I}}  =I$.
\end{remark}

\begin{example}  \label{mixed2}
We return to the ideal
$I$ of
Example \ref{mixed}.   For this ideal we have
\[\widehat{I} = \left\langle
\frac{x_1\cdots x_5}{x_1x_4},
\frac{x_1\cdots x_5}{x_2x_5},
\frac{x_1\cdots x_5}{x_1x_2x_3},
\frac{x_1\cdots x_5}{x_3x_4x_5} \right\rangle
= \langle x_2x_3x_5,x_1x_3x_4,x_4x_5,x_1x_2 \rangle.
\]
\end{example}

The next lemma is key to understanding how the
$f$-vector behaves under the duality.

\begin{lemma}\label{bijection}
Let $I \subseteq R = k[x_1,\ldots,x_n]$ be 
a square-free monomial ideal.  For all integers
$j = -1,\ldots,n-1$, there is a bijection
\[\left\{m \in R_{j+1} ~|~ \mbox{$m$ a square-free monomial
that divides some $p \in \mathcal{G}(I)$} \right\} \]
\[\leftrightarrow
\left\{m \in \widehat{I}_{n-j-1}~|~ \mbox{$m$ a square-free
monomial} \right\}.
\]
\end{lemma}

\begin{proof}  Fix a $j \in \{-1,\ldots,n-1\}$ and let 
$A$ denote the first set, and let $B$ denote the second set.  We claim
that the map $\varphi:A \rightarrow B$ given by 
\[\varphi(m) = \frac{x_1x_2\cdots x_n}{m}\]
gives the desired bijection.    This map is defined because
if $m \in A$, there is a generator $p \in \mathcal{G}(I)$ such that
$m|p$.  But that then means that $\frac{x_1\cdots x_n}{p}$ divides
$\varphi(m) = \frac{x_1\cdots x_n}{m}$, and consequently,
$\varphi(m) \in \widehat{I}$.  Moreover, since $\deg(m) =j+1$,
we have $\deg(\varphi(m))= n-j-1$.  Finally, since
$m$ is a square-free monomial, so is $\varphi(m)$.

It is immediate that the map is injective.
For surjectivity,
let $m \in B$.   It suffices to show that the square-free monomial
$m' = \frac{x_1 \cdots x_n}{m} \in A$ since $\varphi(m') = m$.  
By our construction of $m'$ it follows that $\deg(m') = j+1$.  Also, because
$m \in B$, there is some $p \in \mathcal{G}(I)$ such
that $\frac{x_1\cdots x_n}{p}$ divides $m$.  But this then means that
$m'$ divides $p$, i.e., $m' \in A$.
\end{proof}
\begin{remark}
Using the notation introduced before Lemma \ref{partitionchar},
Lemma \ref{bijection} gives a bijection
between $B_{j+1}(I) \sqcup C_{j+1}(I)$ and
$C_{n-j-1}(\widehat{I}) \sqcup D_{n-j-1}(\widehat{I})$
for all $j = -1,\ldots,n-1$.
\end{remark}

Lemma \ref{bijection} can be used to
relate the $f$-vectors of $\delta_\mathcal{N}(I),
\delta_\mathcal{F}(I), \delta_\mathcal{N}(\widehat{I})$, and
$\delta_\mathcal{F}(\widehat{I})$.

\begin{corollary}\label{maincor}
Let $I \subseteq R = k[x_1,\ldots,x_n]$ be a square-free monomial
ideal.  
\begin{enumerate}
\item[$(i)$] If $f(\delta_\mathcal{F}(I)) = (f_{-1},f_0,\ldots,f_d)$, then
\[f(\delta_\mathcal{N}(\widehat{I})) = 
\left(\binom{n}{0}- f_{n-1},\ldots,\binom{n}{i}-f_{n-i-1},\ldots,
\binom{n}{n-1}-f_0,\binom{n}{n}-f_{-1}\right).\]
\item[$(ii)$] If $f(\delta_\mathcal{N}(I)) = (f_{-1},f_0,\ldots,f_d)$, then
\[f(\delta_\mathcal{F}(\widehat{I})) = 
\left(\binom{n}{0}- f_{n-1},\ldots,\binom{n}{i}-f_{n-i-1},\ldots,
\binom{n}{n-1}-f_0,\binom{n}{n}-f_{-1}\right).\]
\end{enumerate}
In both cases, $f_i = 0$ if $i > d$.
\end{corollary}

\begin{proof}
$(i)$.   Fix some $j \in \{-1,0,\ldots,n-1\}$.   
By Remark \ref{monomialpov} and Lemma \ref{bijection}
we have
\begin{eqnarray*}
f_j(\delta_\mathcal{F}(I)) & = & 
\#\left\{m \in R_{j+1} ~|~ \mbox{$m$ a square-free monomial
that divides some $p \in \mathcal{G}(I)$} \right\} 
\\
& = & \#\left\{m \in \widehat{I}_{n-j-1}~|~ \mbox{$m$ a square-free
monomial} \right\} \\
& = & \#\left\{m \in R_{n-j-1}~|~ \mbox{$m$ a square-free
monomial}\right\} \\
&&\hspace{.4cm}-~ \#\left\{m \not\in \widehat{I}_{n-j-1}~|~ \mbox{$m$ a square-free
monomial}\right\} \\
& = & \binom{n}{n-j-1} - f_{n-j-2}(\delta_\mathcal{N}(\widehat{I})).
\end{eqnarray*}
Rearranging, and letting $l = n-j-2$ gives
\[f_{l}(\delta_\mathcal{N}(\widehat{I})) = 
\binom{n}{l+1} - f_{n-l-2}(\delta_\mathcal{F}(I))
~~\mbox{for $l=-1,\ldots,n-1$,}\]
as desired.

$(ii)$ The proof is similar to $(i)$.  Indeed, if we replace $I$ with $\widehat{I}$
we showed that 
\[f_j(\delta_\mathcal{F}(\widehat{I})) = 
\binom{n}{n-j-1} - f_{n-j-2}(\delta_\mathcal{N}(I)) =
\binom{n}{j+1} -f_{n-j-2}(\delta_\mathcal{N}(I)) \]
for all $j \in \{-1,0,\ldots,n-1\}$.
\end{proof}

We end this section with some consequences related to the 
Kruskal-Katona theorem;  although we do not use this
result in the sequel, we feel it is of independent interest.

We follow
the notation of Herzog-Hibi 
\cite[Section 6.4]{herzog}.   The 
{\it Macaulay expansion} of $a$ with respect to $j$ is the expansion
\[a = \binom{a_j}{j}+\binom{a_{j-1}}{j-1}+\cdots + \binom{a_k}{k}\]
where $a_j > a_{j-1} > \cdots > a_k \geq k \geq 1$.  
This expansion is unique (see
\cite[Lemma 6.3.4]{herzog}).
For a fixed $a$ and $j$, we use the Macaulay expansion
of $a$ with respect to $j$ to define
\[a^{(j)} = \binom{a_j}{j+1}+\binom{a_{j-1}}{j}+\cdots + \binom{a_k}{k+1}.\]
Kruskal-Katona's theorem  \cite{ka,kr} then classifies what vectors
can be the $f$-vector of a simplicial complex using
the Macaulay expansion operation.   This equivalence, plus
two new equivalent statements which use the complementary
dual, are given below.

\begin{theorem}\label{kk}

Let $(f_{-1},f_0,\ldots,f_d) \in \mathbb{N}_+^{d+2}$ with
$f_{-1} = 1$.
   Then the following are equivalent:
    \begin{enumerate}
    \item[$(i)$] $(f_{-1},f_0,f_1,\ldots,f_d)$ is the $f$-vector of a
      simplicial complex on $n=f_0$ vertices.
      \item[$(ii)$] $f_t \leq f_{t-1}^{(t)}$ for all $1 \leq t \leq d$.
      \item[$(iii)$] 
\[\left(\binom{n}{0}- f_{n-1},\ldots,\binom{n}{i}-f_{n-i-1},\ldots,
\binom{n}{n-1}-f_0,\binom{n}{n}-f_{-1}\right)\]
is the $f$-vector of a simplicial complex on 
$\binom{n}{1}-f_{n-2}$ vertices (where $f_i = 0$ if $i > d$).
  \item[$(iv)$] \[\binom{n}{t+1} - \left[\binom{n}{t+2} - f_{t+1}\right]^{(n-t-2
)} \leq f_t  ~~~~~~~
\hspace{.3cm}\mbox{for all $0 \leq t \leq d-1$}.\] 
      \end{enumerate}
    \end{theorem}

\begin{proof}
$(i) \Leftrightarrow (ii)$. This equivalence is 
the Kruskal-Katona theorem (see \cite{ka,kr}).

$(i) \Leftrightarrow (iii)$.  This equivalence follows from
Corollary \ref{maincor} and the Kruskal-Katona equivalence of
$(i) \Leftrightarrow (ii)$.  In particular, one lets 
$I$ be the square-free monomial ideal with 
$f(\delta_\mathcal{N}(I)) = (f_{-1},f_0,\ldots,f_d)$, and then
one uses Corollary \ref{maincor} to show that $(iii)$ is a valid
$f$-vector.  The duality of $I$ and $\widehat{I}$ is used to show the
reverse direction.

$(iii) \Leftrightarrow (iv)$.
Corollary \ref{maincor} and the equivalence of $(i) \Leftrightarrow (ii)$
implies that the vector of $(iii)$ is an $f$-vector 
of a simplicial complex if and only if, 
for each $0\leq i \leq n-2$,
\[\binom{n}{i+2} - f_{n-i-3} \leq \left[\binom{n}{i+1}- f_{n-i-2}\right]^{(i+1)}.\]
The result now follows if we take $i = n-3-t$ and rearrange the above equation.
\end{proof}

\begin{discussion}
It is prudent to make some observations 
about the Alexander dual, although this material
is not required for our paper.
Recall that for any simplicial complex $\Delta$ on a vertex set
$X$, the 
{\it Alexander dual} of $\Delta$ is the simplicial 
complex on $X$ given by  
\[\Delta^\vee = \{F \subseteq X ~|~ X \setminus F \not\in \Delta\}.\]
It is well-known (for example, see \cite[Corollary 1.5.5]{herzog}) that
if $\Delta = \langle F_1,\ldots,F_s \rangle$, then 
$\mathcal{N}(\Delta^\vee)$, the non-face
ideal of $\Delta^\vee$ (i.e., the ideal generated by
the square-free monomials $x_{i_1}\cdots x_{i_j}$ where
$\{x_{i_1},\ldots,x_{i,j}\} \not\in \Delta$) is given by
\[\mathcal{N}(\Delta^\vee) = \langle m_{F_1^c},\ldots,m_{F_s^c} \rangle\]
where $m_{F_i^c} = \prod_{x \in F_i^c} x$ with $F_i^c = X \setminus F_i$.
But we can also write $m_{F_i^c} = (\prod_{x\in X} x )/m_{F_i}
= \frac{x_1\cdots x_n}{m_{F_i}}$.  Now tracing through the definitions,
if $I$ is square-free monomial ideal, then
\[\widehat{I} = \mathcal{N}((\delta_{\mathcal{F}}(I))^\vee),\]
i.e., the complementary dual of $I$ is the non-face ideal of 
the Alexander dual of the facet complex of $I$.  This, in turn, implies
that $\delta_\mathcal{N}(\widehat{I}) = \delta_\mathcal{N}(\mathcal{N}((\delta_{\mathcal{F}}(I))^\vee) = (\delta_{\mathcal{F}}(I))^\vee$.
\end{discussion}


\section{$f$-ideals and applications}
\label{fvector}

We use the tools of the previous sections to prove our main
theorem about $f$-ideals and to deduce some new consequences
about this class of ideals.
Our main theorem is an immediate application of Corollary
\ref{maincor}.

\begin{theorem}\label{maintheorem2}
Let $I$ be a square-free monomial ideal of $R = k[x_1,\ldots,x_n]$.
Then $I$ is an $f$-ideal if and only if $\widehat{I}$ is an
$f$-ideal.
\end{theorem}

\begin{proof}
Suppose $f(\delta_\mathcal{N}(I))= f(\delta_\mathcal{F}(I)) = 
(f_{-1},f_0,\ldots,f_d)$.
Then by Corollary \ref{maincor}, both 
$\delta_\mathcal{N}(\widehat{I})$ and $\delta_\mathcal{F}(\widehat{I})$ will 
have the same $f$-vector.  For the reverse direction, simply
replace $I$ with $\widehat{I}$ and use the same corollary.
\end{proof}

\begin{remark}
Theorem \ref{maintheorem2} was first proved by the first author 
(see \cite{B}) using the  characterization of $f$-ideals of \cite{guo}.  
The proof presented here avoids the machinery of \cite{guo}.
\end{remark}

\begin{example}
In Example \ref{mixed2} we computed the ideal $\widehat{I}$ of 
the ideal $I$ in Example \ref{mixed}.  By
Theorem \ref{maintheorem2} the ideal $\widehat{I}$ is an $f$-ideal.  
Indeed, the simplicial complexes
 $\delta_{\mathcal{N}}(\widehat{I})$
and $\delta_{\mathcal{N}}(\widehat{I})$ are given in Figure \ref{figure2},
and both simplicial complexes have $f$-vector $(1,5,8,2)$.   
Note that $f$-vector of
$\delta_\mathcal{N}(I)$ and $\delta_\mathcal{F}(I)$ was $(1,5,8,2)$, 
so by Corollary \ref{maincor}
\begin{eqnarray*}
f(\delta_{\mathcal{F}}(\widehat{I})) = f(\delta_\mathcal{N}(\widehat{I}))& = &
\left(\binom{5}{0}-0,\binom{5}{1}-0,\binom{5}{2}-2,\binom{5}{3}-8,
\binom{5}{4}-5,\binom{5}{5}-1\right)\\
& = & (1,5,8,2).
\end{eqnarray*}

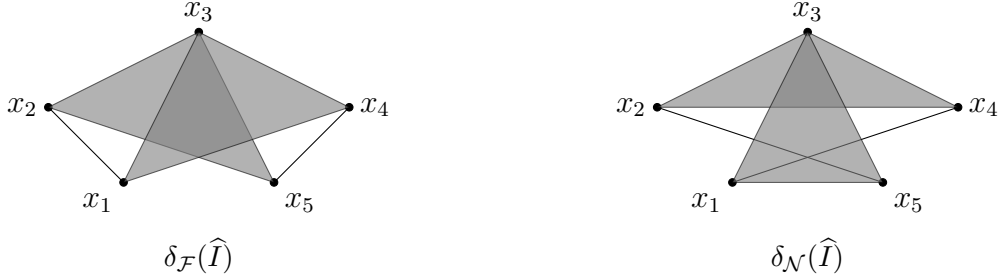
\begin{figure}[H]
    \centering
    \begin{subfigure}[h]{0.50\textwidth}
        \centering
        \begin{tikzpicture}
\draw[fill] (0,0) circle [radius=0.05];
\node [left] at (0,0) {$x_2$};
\draw[fill] (1,-1) circle [radius=0.05];
\node [below left] at (1,-1) {$x_1$};
\draw[fill] (2,1) circle [radius=0.05];
\node [above] at (2,1) {$x_3$};
\draw[fill] (3,-1) circle [radius=0.05];
\node [below right] at (3,-1) {$x_5$};
\draw[fill] (4,0) circle [radius=0.05];
\node [right] at (4,0) {$x_4$};
\path (3,-1) edge (4,0);
\path (1,-1) edge (0,0);
\draw[fill=gray, opacity=0.6] (0,0) -- (3,-1) -- (2,1) -- (0,0);
\draw[fill=gray, opacity=0.6] (1,-1) -- (2,1) -- (4,0) -- (1,-1);

\end{tikzpicture}

        \caption*{$\delta_\mathcal{F} (\widehat{I})$}
        
    \end{subfigure}
    ~ 
    \begin{subfigure}[h]{0.5\textwidth}
        \centering
        \begin{tikzpicture}
\draw[fill] (0,0) circle [radius=0.05];
\node [left] at (0,0) {$x_2$};
\draw[fill] (1,-1) circle [radius=0.05];
\node [below left] at (1,-1) {$x_1$};
\draw[fill] (2,1) circle [radius=0.05];
\node [above] at (2,1) {$x_3$};
\draw[fill] (3,-1) circle [radius=0.05];
\node [below right] at (3,-1) {$x_5$};
\draw[fill] (4,0) circle [radius=0.05];
\node [right] at (4,0) {$x_4$};
\path (3,-1) edge (0,0);
\path (1,-1) edge (4,0);
\draw[fill=gray, opacity=0.6] (0,0) -- (4,0) -- (2,1) -- (0,0);
\draw[fill=gray, opacity=0.6] (1,-1) -- (2,1) -- (3,-1) -- (1,-1);

\end{tikzpicture}

        \caption*{$\delta_\mathcal{N}(\widehat{I})$}
         \end{subfigure}
    \caption{Facet and non-face complexes of $\widehat{I}=\langle x_1x_2, x_4x_5, x_1x_3x_4, x_2x_3x_5 \rangle$.}\label{figure2}
\end{figure}
\end{example}

Theorem \ref{maintheorem2} implies that $f$-ideals
come in ``pairs".   This observation allows us to extend
many known results about $f$-ideals to their complementary duals.  
For example, we can now classify the $f$-ideals that
are equigenerated in degree $n-2$.

\begin{theorem}   Let $I$ 
be a square-free monomial of $k[x_1,\ldots,x_n]$ equigenerated
in degree $n-2$.   Then the following are equivalent:
\begin{enumerate}
\item[$(i)$]  $I$ is an $f$-ideal.
\item[$(ii)$] $\widehat{I}$ is an $f$-ideal.
\item[$(iii)$] $\widehat{I}$ is an unmixed ideal
of height $n-2$ (i.e., all of the associated primes of
$I$ have height $n-2$)  with $p = \frac{1}{2}\binom{n}{2}$.
\end{enumerate}
\end{theorem}

\begin{proof}
The equivalence of $(i)$ and $(ii)$ is 
Theorem \ref{maintheorem2}.   Because
the ideal $\widehat{I}$ is equigenerated in
degree two,  the equivalence of $(ii)$ and $(iii)$ is 
\cite[Theorem 3.5]{abbasi}.
\end{proof}

Following Guo, Wu, and Liu \cite{guo}, let $V(n,d)$
denote the set of $f$-ideals of $k[x_1,\ldots,x_n]$
that are equigenerated in degree $d$.   We 
can now extend Guo, et al.'s results.

\begin{theorem}\label{vnd} Using the notation above, we have  
\begin{enumerate}
\item[$(i)$] For all $1 \leq d \leq n-1$, 
$|V(n,d)| = |V(n,n-d)|$.
\item[$(ii)$] If $n \neq 2$, then
$V(n,1) = V(n,n-1) = \emptyset$.  If $n = 2$, then
$|V(2,1)| = 2$.
\item[$(iii)$] $V(n,n-2) \neq \emptyset$ if and only
if $n \equiv 0$ or 1$ \pmod{4}$.
\end{enumerate}
\end{theorem}

\begin{proof}
$(i)$ By Theorem \ref{maintheorem2}, 
the complementary dual gives a bijection
between the sets $V(n,d)$ and $V(n,n-d)$.  

$(ii)$ Suppose that $I \in V(n,1)$, i.e., $I$
is an $f$-ideal generated by a subset of the variables.
So, the facets of $\delta_\mathcal{F}(I)$ are vertices, while
$\delta_\mathcal{N}(I)$ is a simplex.  Then $f_0(\delta_\mathcal{F}(I))$, the 
number of variables that generate $I$, must be
the same as $f_0(\delta_\mathcal{N}(I))$,
the number of variables not in $I$.
This implies that $n$ cannot be odd.   
Furthermore, if $n \geq 4$ is even,
then $\dim \delta_\mathcal{F}(I) = 0$, but
$\dim \delta_\mathcal{N}(I) = \frac{n}{2}-1 \geq 1$,
contradicting the fact that $I$ is an $f$-ideal.

When $n=2$, then $I_1=\langle x_1 \rangle$ and 
$I_2 = \langle x_2 \rangle$ are $f$-ideals of
$k[x_1,x_2]$.

$(iii)$ By $(i)$, $|V(n,n-2)| \neq 0$ if and only
if $|V(n,2)| \neq 0$.  Now \cite[Proposition 3.4]{guo}
shows that $V(n,2) \neq \emptyset$ if and only
if $n \equiv 0,1 \pmod{4}$.
\end{proof}

\begin{remark}  \cite[Proposition 4.10]{guo} gives
an explicit formula for $|V(n,2)|$, which we will not present
here.  So by Theorem \ref{vnd} $(i)$ there is an 
explicit formula for $|V(n,n-2)|$.
\end{remark}

We now explore some necessary conditions on the $f$-vector
of $\delta_\mathcal{N}(I)$ (equivalently, $\delta_\mathcal{F}(I)$) when
$I$ is an $f$-ideal.   We also give a necessary condition on 
the generators of an $f$-ideal.    As we shall see, 
Theorem \ref{maintheorem2} plays a role in some of our proofs.

We first recall some notation.  
If $I \subseteq R$ is a square-free monomial ideal, then we let
\[\alpha(I) = \min\{\deg(m) ~|~ m \in \mathcal{G}(I)\} 
~~\mbox{and}~~ 
\omega(I) = \max\{\deg(m)~|~ m \in \mathcal{G}(I)\}.
\]
We present some conditions on the $f$-vector;
some of these results were known.

\begin{theorem}
Suppose that $I$ is an $f$-ideal in
$R = k[x_1,\ldots,x_n]$ with associated $f$-vector
$f = f(\delta_{\mathcal{F}}(I)) = f(\delta_{\mathcal{N}}(I))$.
Let $\alpha = \alpha(I)$ and $\omega = \omega(I)$.
Then
\begin{enumerate}
\item[$(i)$] $f_i = \binom{n}{i+1}$ for $i=0,\ldots,\alpha-2$.
\item[$(ii)$] $f_{\alpha-1} \geq \frac{1}{2}\binom{n}{\alpha}$.
\item[$(iii)$] $f_{\omega-1} \leq \frac{1}{2}\binom{n}{\omega}$.
\item[$(iv)$] if $\alpha = \omega$ (i.e., $I$ is 
equigenerated), then $f_{\alpha-1} = \frac{1}{2}\binom{n}{\alpha}$.
\item[$(v)$] $\dim \delta_{\mathcal{F}}(I) =
\dim \delta_{\mathcal{N}}(I) = \omega-1 \leq n-2$.
\end{enumerate}
\end{theorem}

\begin{proof}
$(i)$  See \cite[Lemma 3.7]{anwar}.

$(ii)$ If $I$ is generated by monomials of degree 
$\alpha$ or larger, then \eqref{partition} 
becomes
\[M_\alpha = A_\alpha(I) \sqcup B_\alpha(I) \sqcup C_\alpha(I)
~~\mbox{since $D_\alpha(I) = \emptyset$}.\]  Suppose
$f_{\alpha-1} < \frac{1}{2}\binom{n}{\alpha}$.
Because $f_{\alpha-1}(\delta_\mathcal{N}(I)) = |A_\alpha(I)|+|B_\alpha(I)|$,
we have $|C_\alpha(I)| > \frac{1}{2}\binom{n}{\alpha}$.
But since $I$ is an
$f$-ideal, by  Lemma \ref{partitionchar} we have
\[\frac{1}{2}\binom{n}{\alpha}>  f_{\alpha-1}(\delta_\mathcal{N}(I)) \geq |A_\alpha(I)| = |C_\alpha(I)| > \frac{1}{2}\binom{n}{\alpha}.\]
We now have the desired contradiction.

$(iii)$ Suppose that $f_{\omega-1}(\delta_{\mathcal{F}}(I)) > \frac{1}{2}\binom{n}{\omega} =\frac{1}{2}\binom{n}{n-\omega}$.   Since $\omega = \omega(I)$,
we must have that $\alpha(\widehat{I}) = n - \omega$.  Since $\widehat{I}$ is also an $f$-ideal
by Theorem \ref{maintheorem2}, $(ii)$ implies that $f_{n-\omega-1}(\delta_{\mathcal{F}}(\widehat{I})) \geq
\frac{1}{2}\binom{n}{n-\omega}$.  But by Corollary \ref{maincor},
and since $\widehat{I}$ is an $f$-ideal,   
\[f_{n-\omega-1}(\delta_{\mathcal{F}}(\widehat{I}))=
f_{n-\omega-1}(\delta_\mathcal{N}(\widehat{I}))
= \binom{n}{n-\omega} - f_{\omega-1}(\delta_{\mathcal{F}}(I)) < \frac{1}{2}\binom{n}{n-\omega}.\]
This gives the desired contradiction.

$(iv)$  We simply combine the inequalities of $(ii)$ and $(iii)$.

$(v)$
Since $\omega = \omega(I)$, there is a generator $m$ of $I$ of degree 
$\omega$, and furthermore,
every other generator has smaller degree.  So, the 
facet of $\delta_\mathcal{F}(I)$ of largest dimension has dimension
$\omega-1$.   Since $I$ is an $f$-ideal, this also forces 
$\delta_\mathcal{N}(I)$ to have a facet of
dimension of $\omega-1$.  Note that $\omega(I) \leq n-1$ since 
no $f$-ideal has $x_1\cdots x_n$ as a generator.
\end{proof}

Our final result shows that if $I$ is not an equigenerated $f$-ideal, 
then in some cases we can deduced the existence of generators of
other degrees.

\begin{theorem}\label{generators}  
Suppose that $I$ is an $f$-ideal of $k[x_1,\ldots,x_n]$ 
with $\alpha = \alpha(I) < \omega(I) = \omega$,
and let $f = f(\delta_{\mathcal{F}}(I)) = f(\delta_{\mathcal{N}}(I))$.
\begin{enumerate}
\item[$(i)$]  If $f_{\alpha-1} > \binom{n}{\alpha} - n+\alpha$, then $I$ also 
has a generator of degree $\alpha+1$.
\item[$(ii)$] If $f_{\omega-1} < \omega$, then $I$ also has a generator of 
degree $\omega-1$.
\end{enumerate}
\end{theorem}

\begin{proof}
To prove $(i)$, it is enough to prove $(ii)$ and apply Theorem 
\ref{maintheorem2}.  Indeed, suppose
that $f_{\alpha-1} > \binom{n}{\alpha} - n + \alpha$.  Then the ideal 
$\widehat{I}$ is an $f$-ideal with
$\omega(\widehat{I}) = n-\alpha$ and
\[f_{\omega(\widehat{I})-1}(\delta_\mathcal{F}(\widehat{I})) = 
\binom{n}{n-\alpha} - f_{\alpha-1} < n-\alpha = 
\omega(\widehat{I}).\]  So by $(ii)$
the ideal $\widehat{I}$ will have a generator of degree 
$\omega(\widehat{I})-1$, which implies that
$I$ has a generator of degree $\alpha+1$.

$(ii)$ Note that if $\alpha = \omega-1$, then the conclusion immediately 
follows.  So suppose that $\alpha < \omega-1$.
We use the partition \eqref{partition}.  Since 
$I$ is generated in degrees $\leq \omega$, we have
$B_\omega(I) = \emptyset$.  It then follows by
Lemma \ref{partitionchar} and Remark
\ref{monomialpov} that \[f_{\omega-1} = |A_\omega(I)| = 
|C_\omega(I)| < \omega.\]

Now suppose that $I$ has no generators of degree
$\omega-1$.  So $|C_{\omega-1}(I)| = 0$, and 
consequently, $|A_{\omega-1}(I)| = 0$ because
$I$ is an $f$-ideal.   Because $\alpha < \omega-1$,
we have $D_{\omega-1}(I) \neq \emptyset$.
Then, 
again by Lemma \ref{partitionchar} and Remark
\ref{monomialpov}, we must have
$f_{\omega-2} = |B_{\omega-1}(I)|.$   Let $m \in A_\omega(I)$.
After relabeling, we can assume that $m = x_1x_2\cdots x_\omega$.   
Note that $m/x_i \not\in I$ for $i=1,\ldots,
\omega$.  Indeed, if $m/x_i \in I$, this implies that $m \in I$,
contradicting the fact that all elements of $A_\omega(I)$
are not in $I$.   So $m/x_i \in B_{\omega-1}(I)$ for
all $i$.  By definition, every element of $B_{\omega-1}(I)$
must divide an element of $C_{\omega}(I)$ (since
$B_\omega(I) = \emptyset$).  Because $|C_\omega(I)| < \omega$,
there is one monomial $z \in C_\omega(I)$ such that 
$m/x_i$ and $m/x_j$ both divide $z$.  But since $\deg z =
\omega$, this forces $m = z$.  We now arrive at a 
contradiction since $m \in A_\omega(I) \cap C_\omega(I)$,
but these two sets are disjoint.
\end{proof}

\begin{remark}
The ideal $I$ of Example \ref{mixed} is an $f$-ideal 
with $\alpha = \alpha(I) = 2$, and $f_{2-1} = 8 > \binom{5}{2}-5+2 = 7$.
So by Theorem \ref{generators}, the ideal $I$ should have a generator
of degree $\alpha + 1 = 3$, which it does.  Alternatively, we could
have deduced that $I$ has a generator of degree 2 from the fact that
$\omega(I) = 3$ and $f_{3-1} = 2$.

In our computer experiments, we only found $f$-ideals which had either
$\alpha(I) = \omega(I)$, i.e., the $f$-ideals were equigenerated, or
$\alpha(I)+1 = \omega(I)$.    It would be interesting to determine
the existence of $f$-ideals with the property that $\alpha(I)+d = \omega(I)$
for any $d \in \mathbb{N}$.   Theorem \ref{generators} would imply
a necessary condition on the generators of these ideals. 
\end{remark}


\bibliographystyle{plain}

\end{document}